\documentclass[12pt]{amsart}

\RequirePackage[nonfrench]{kotex}
\usepackage{kotex}

\usepackage{epsfig}
\usepackage{amsfonts}
\usepackage{amssymb}
\usepackage{amsmath}
\usepackage{amsthm}
\usepackage{amscd}
\usepackage{amstext}
\usepackage{latexsym}

\usepackage{enumerate}

\newcommand{\CC}{\mathbf{C}}

\newcommand{\RR}{\mathbf{R}}

\newcommand{\JJ}{\mathcal{J}}

\newcommand{\OO}{\mathcal{O}}

\newcommand{\al}{\alpha}

\newcommand{\mfa}{\mathfrak{a}}
\newcommand{\mfb}{\mathfrak{b}}

\newcommand{\mfm}{\mathfrak{m}}

\newcommand{\mfab}{\mfa_\bullet}
\newcommand{\mfbb}{\mfb_\bullet}

\newcommand{\ep}{\epsilon}

\newcommand{\ld}{\lambda}

\newcommand{\qa}{\quad}
\newcommand{\vp}{\varphi}

\newcommand{\noi}{\noindent}

\providecommand{\abs}[1]{\left|#1\right|}

\theoremstyle{plain}

\newtheorem{theorem}{Theorem}[section]
\newtheorem{cor}[theorem]{Corollary}
\newtheorem{prp}[theorem]{Proposition}

\newtheorem{dfn}[theorem]{Definition}

\newtheorem{thm}[theorem]{Theorem}

\newtheorem{lemma}[theorem]{Lemma}
\newtheorem{corollary}[theorem]{Corollary}
\newtheorem{proposition}[theorem]{Proposition}
\newtheorem{prop}[theorem]{Proposition}
\newtheorem{definition}[theorem]{Definition}

 \newtheorem{example}[theorem]{\textnormal{\textbf{Example}}}

\theoremstyle{remark}
\newtheorem{remark1}[theorem]{Remark}

\DeclareMathOperator{\PSH}{PSH}

\DeclareMathOperator{\ord}{ord}

\begin{document}

\title[Asymptotic multiplicities and Monge-Amp\`ere masses]{Asymptotic multiplicities and Monge-Amp\`ere masses \\
 \small (with an appendix by S\'ebastien Boucksom)}

\keywords{Graded systems of ideals, Samuel multiplicity, plurisubharmonic functions, Monge-Amp\`ere mass, multiplier ideals, Demailly approximation}


\author{Dano Kim and Alexander Rashkovskii }

\date{}

\begin{abstract}

\noindent   Ein, Lazarsfeld and Smith asked whether `equality' holds between two Samuel type asymptotic multiplicities for a graded system of zero-dimensional ideals on a smooth complex variety.  We find a connection of this question  to complex analysis by showing that the `equality' is equivalent to a particular case of Demailly's strong continuity property on the convergence of residual Monge-Amp\`ere masses under approximation of plurisubharmonic functions.  On the other hand, in an appendix of this paper, S\'ebastien Boucksom gives an algebraic proof of the `equality' in general, using the intersection theory of b-divisors. We then use these to show that  Demailly's strong continuity holds for a new important class of plurisubharmonic functions. 

\end{abstract}

\maketitle


\setcounter{tocdepth}{1}  

\section{Introduction}

Let $X$ be an irreducible smooth complex variety  of dimension $n \ge 1$. A graded system of ideals in $\OO_X$  is a sequence of coherent ideal sheaves  $\mfab = (\mfa_m)_{m \ge 1}$  on $X$ satisfying $\mfa_m \cdot \mfa_k \subset \mfa_{m+k}$ for all $m, k \ge 1$.  Generalizing base ideals of linear systems on smooth projective varieties, they arise naturally from many contexts in algebra and geometry : see \cite[(1.2)]{ELS01},  \cite[\S 2.4.B]{L} for their rich examples and properties.

  Let $p \in X$ be a point. Assume that for every $m \ge 1$,  $\mfa_m$ is zero-dimensional at $p$, i.e. its zero set is equal to $\{ p \}$. In particular, $\mfa_m \neq \{0\}$ and  each $\mfa_m$ is $\mathfrak{m}$-primary  where $\mfm$ is the maximal ideal of $p$. Let $\mfbb$ be the sequence of asymptotic multiplier ideals of $\mfab$ (Definition~\ref{am}). Due to  \cite{ELS03}, \cite{M02}, the  following limits exist and define asymptotic Samuel multiplicities of $\mfab$ and of $\mfbb$ respectively  :

\begin{equation}\label{ealimit}
  e(\mfab) := \lim_{m \to \infty} \frac{e(\mfa_m)}{m^n} \;\qa \qa \text{           and           } \qa \quad e(\mfbb) := \lim_{m \to \infty} \frac{e(\mfb_m)}{m^n}   
 \end{equation}

\noi where $e(\cdot)$ denotes the  Samuel multiplicity of an $\mfm$-primary ideal,  which is equal to the intersection number of general $n = \dim X$ elements in $\mfa$ (cf. \cite[\S 1.6.B]{L}). While the inequality $e(\mfab) \ge e(\mfbb)$ is clear, the equality 

\begin{equation}\label{eab}
e(\mfab) = e(\mfbb)
\end{equation} 

\noi was raised as a question in  \cite[p.432]{ELS03} (cf. \cite[Question 2.20]{M02}). It can be viewed as saying that singularities of the multiplier ideals $\mfbb$ are close enough to those of $\mfab$ in a sense. The equality was previously known in some cases:  when $\mfab$ are the valuation ideals of a quasimonomial valuation \cite{ELS03}, when $\mfab$ are monomial ideals \cite{M02} and  when $\mfab$ are associated (in the sense of \eqref{gdd}) to a maximal tame psh weight \cite[Thm. 5.8]{R13}.

While the equality \eqref{eab} in general will be shown by S\'ebastien Boucksom in the appendix of this paper (see Theorem~\ref{eaeb}), we now turn to the main theme of this paper: singularities of plurisubharmonic functions.

 We will show that, in fact,  $e(\mfab) = e(\mfbb)$ is equivalent to a special case of the `strong continuity' of Monge-Amp\`ere operators with respect to  Demailly approximation \cite{D92} of a plurisubharmonic (psh for short) function.  We refer to \cite[Thm. 13.2]{D11} for the definition and existence of Demailly approximation.  The key feature of the Demailly approximation of a psh function $\vp$ is that it is given by a sequence of psh functions $\vp_m$ with analytic singularities  given by the multiplier ideal $\JJ(m\vp)$ (divided by $m$). Following the terminology used in \cite{D12}, we have

\begin{definition}[Demailly \cite{D12}]\label{dem}

 Let $\vp$ be a psh function with isolated singularities at $0 \in D$, an open ball in $\CC^n$. Let $\{ \vp_m \}_{m \ge 1}$ be the Demailly approximation sequence of $\vp$. We will say that \textbf{\emph{Demailly's strong continuity}} holds for $\vp$ if we have the convergence of the $n$-th Lelong numbers   $\; L_n (\vp_m, 0) \to L_n (\vp, 0) $ as $m \to \infty$.

\end{definition}

 Here, a psh function with \emph{isolated singularities} (also called a psh weight in this paper) at a point $0 \in D$ is  a psh function that is locally bounded outside the point $0$.   The $n$-th Lelong number $L_n (\vp, 0) $ is also known as the residual Monge-Amp\`ere mass $(dd^c \vp)^n (\{0\})$ of the current $(dd^c \vp)^n$ which can be defined in this case  due to \cite{D87} (cf. \cite{D93}, \cite{DX}). Connection with Samuel multiplicities arises from the fact due to \cite[Lem. 2.1]{D09} that when $\vp$ has analytic singularities given by a zero-dimensional ideal $\mfa$ at $0$, $L_n (\vp, 0)$ is equal to the Samuel multiplicity of $\mfa$.

Demailly~\cite{D12} asked whether the strong continuity holds for every psh function with isolated singularities. When it holds for  $\vp$, that would essentially mean that higher Lelong numbers of $\vp$ are well controlled by the multiplier ideals of $\vp$ (as is the case already for the first, i.e. usual Lelong number of $\vp$, cf. \cite{D92}). 
Previous positive results  include the cases  when  $\vp$ is tame by \cite{BFJ} and more generally when $\vp$ is asymptotically analytic by  \cite{R13} (see Definition~\ref{5psh}).    
On the other hand, due to recent work of C. Li \cite{L21}, cf. \cite{TV21}, it is now known that Demailly's strong continuity does not hold for all psh weights. 
 
  This makes it all the more intriguing to understand which psh weights have such an important property. 
 In this paper, we will show that Demailly's strong continuity holds for a new important class of psh weights, which we will describe. Together with the above results \cite{BFJ}, \cite{R13}, these will represent by far the most general positive results on Demailly's strong continuity.

 Given a graded system of ideals $\mfab$ on $X$, there are two natural ways to associate psh functions on a ball \footnote{or more generally  a  bounded hyperconvex domain, either of which can  of course be viewed as sitting in $X$}  to $\mfab$ which we will call as  the Green function $G_{\mfa_\bullet}$ and (non-unique)  Siu functions $\vp_{\mfab}$ respectively. The former is defined by a natural pointwise supremum construction while the latter is defined as an infinite series involving members of $\mfab$ (see \S \ref{SG}).   Both of them can be seen as local analogues of singular hermitian metrics with minimal singularities associated to line bundles on projective manifolds (which in turn are metric analogues of base loci of line bundles) cf. \cite{D11}, \cite{L}.  
 
 From the viewpoint of algebraic geometry, these are some of the most important psh singularities due to their relations with asymptotic multiplier ideals (cf. \cite[Chapter 11]{L}) : see Theorem~\ref{ksb} and Lemma~\ref{GG} for the relations. 
  We now establish the following key relation between the algebraic and analytic sides. 

\begin{theorem}\label{eaebeq}

Let $\mfab$ be a graded system of $\mfm$-primary ideals on $X$ as above.   The equality $e(\mfa_\bullet) = e(\mfb_\bullet)$ holds if and only if Demailly's strong continuity holds for  the Green function $G_{\mfa_\bullet}$ associated to  $\mfa_\bullet$.

\end{theorem}

  Using Theorem~\ref{eaebeq}, we are able to give an  analytic proof of  the equality  $e(\mfa_\bullet) = e(\mfb_\bullet)$ when the graded system of ideals satisfies a condition which arises from work of K\"uronya and Wolfe~\cite{KW}. We will say that $C >0$ is a K\"uronya-Wolfe constant of $\mfab$ if there exists $D>0$ such that $\mfb_{Cm+D} \subset \mfa_m$ holds for all sufficiently large $m$. By \cite{KW}, such $C >0$ always exists (see Theorem~\ref{KW}).

 \begin{theorem}[=Corollary~\ref{KWC}]\label{kuronya}
 Suppose that a graded system $\mfab$ of $\mathfrak m$-primary ideals  has a K\"uronya-Wolfe constant $C \le 1$.  Then the Green function $G_{\mfab}$ associated to $\mfab$ is tame. Hence  $e(\mfa_\bullet) = e(\mfb_\bullet)$ holds with an analytic proof since Demailly's strong continuity is known for tame psh functions from \cite{BFJ}.
 \end{theorem}

\noi  This generalizes \cite[Prop. 3.11]{ELS03} : see Remark~\ref{els421}.  According to \cite[p.802]{KW}, the condition in this theorem holds for many situations of geometric interest, but not always: see Example~\ref{KW1}.

 The analytic tools we use in obtaining these results come from pluripotential theory (cf. \cite{BT76}, \cite{D85}, \cite{D87}, \cite{Kl91},  \cite{D93}, \cite{C04}, \cite{DH}) and especially its developments in \cite{R06}, \cite{R13} regarding maximality and greenification of psh weights (see \S 3).

 On the other hand, in the appendix of this paper, S\'ebastien Boucksom gives an algebraic proof of the equality  \eqref{eab} in general, thus answering the above question of \cite{ELS03}.  He uses the local intersection theory of nef  b-divisors developed in his joint work with Favre and Jonsson~\cite[\S 4]{BFJ} (also cf. \cite[\S 4.3]{BFF}).

\begin{theorem}[S. Boucksom]\label{eaeb}

 Let $\mfab$ be a graded system of  $\mfm$-primary ideals on $X$ as above. Then we have $e(\mfa_\bullet) = e(\mfb_\bullet)$.

\end{theorem}

While the proof of Theorem~\ref{eaeb} is algebraic in nature, it is remarkable that the theory of \cite{BFJ} was developed in the context of a valuative approach to psh singularities, cf. \cite{JM14} (also cf.  \cite{FJ04}, \cite{FJ05} for such study initiated in dimension $2$ by Favre and Jonsson).  Now combining Theorem~\ref{eaeb} with Theorem~\ref{eaebeq}, we obtain the following new results on Demailly's strong continuity.

 \begin{thm}[=Corollary~\ref{Conjec}]\label{gs}
Demailly's strong continuity holds for Green and Siu functions $G_{\mfa_\bullet}$ and $\vp_{\mfab}$ associated to a graded system of $\mfm$-primary ideals $\mfab$.

 \end{thm}

Indeed, aforementioned previous results on Demailly's strong continuity have been missing these important psh functions : we do not know whether  $G_{\mfa_\bullet}$ and $\vp_{\mfab}$ are always tame or asymptotically analytic. In fact, we give an example of $\vp_{\mfab}$ that is not tame in Example~\ref{KW1}.

  Additionally, we generalize Theorem~\ref{gs} to psh weights with \emph{sup-analytic singularities}, a natural class of psh weights which was introduced in \cite{R13}. They are defined as (up to greenification) the increasing limits of maximal psh weights with analytic singularities.  We show that they indeed include Green and Siu functions in Theorem~\ref{gs} : see 
Corollary~\ref{greensiu}. Then we have the following 

 \begin{theorem}[=Theorem~\ref{sup-an-conj}]\label{supana}

  Suppose that $\vp$ is a psh weight with sup-analytic singularities. Then Demailly's strong continuity holds for $\vp$.

 \end{theorem}

 In particular, this unifies all the currently known positive results (to our knowledge) for Demailly's strong continuity on psh weights since the class also includes psh weights with asymptotically analytic singularities. 

 This paper is organized as follows. In Section 2, we prepare fundamental definitions and properties of psh singularities and psh weights. In Section 3, we discuss Siu and Green   functions associated to a graded system of ideals and prove Theorem~\ref{eaebeq} and Theorem~\ref{gs}. In Section 4, we apply the algebraic result of K\"uronya and Wolfe to our setting and prove Theorem~\ref{kuronya}. In Section 5, we apply Theorem~\ref{eaeb} to study psh weights with sup-analytic singularities and prove Theorem~\ref{supana}. Finally in Section 6 as an appendix, Theorem~\ref{eaeb} is proved by S. Boucksom. 
\\

\noi \textbf{Acknowledgements.}
 We are very grateful to S\'ebastien Boucksom for having interesting and important discussions and   allowing us to include his Theorem~\ref{eaeb} as an appendix in this paper. We also thank Jean-Pierre Demailly for helpful discussions and Philippe Eyssidieux for helpful comments.
D.K. was supported by Basic Science Research Program through NRF Korea funded by the Ministry of Education (2018R1D1A1B07049683).

\section{Plurisubharmonic singularities}
 We refer to \cite[Def. 1.4]{D11}, \cite{Kl91}, \cite{Ki94}, \cite{R00} for the definition of a plurisubharmonic (i.e. psh) function and some of its  basic properties. We also refer to \cite[\S 2]{R13} for some more information regarding   this section.
We first have some \textbf{{conventions and notations}} used in this paper.

\begin{itemize}

\item Since our main interest is on the singularity of a psh function at a point, often it is enough to consider a germ of a psh function (at the origin $0 \in \CC^n$). 

\item When we need to specify a domain of a psh germ at $0$ (for example, when we define a Green function), we  will take  a  bounded hyperconvex domain $D \subset \CC^n$, i.e. a  bounded domain with a negative psh exhaustion function (cf. \cite[p.80]{Kl91}, \cite{C04}), cf. Remark~\ref{weak}.

   \item
 Let $D \subset \CC^n$ be a domain. 
 We denote by $\PSH(D)$ and  $\PSH^-(D)$ the set of psh functions on $D$ and the set of negative (i.e. $\le 0$) psh functions on $D$, respectively.  We denote by $\PSH_p$ the collection of germs of psh functions at a point $p \in D$. 
 
  \item
  
  A \emph{psh weight} at $0$ is a germ of a psh function that is  locally bounded outside $0$. 
  The collection of psh weights (resp. maximal psh weights) at $0$ is denoted by $W_0$ (resp. by $MW_0$). See \S 2.2 and \S 2.3. 

\end{itemize}

\medskip

\subsection{Plurisubharmonic singularities}

Following \cite{D11},  we will say that two psh functions $u$ and $v$ have \emph{equivalent singularities} (and write $u \sim v$) if $u = v + O(1)$. Here $O(1)$ refers to a function which is locally bounded near every point (cf. \cite[Def. 2]{D13}).   
We will say $u$ is \emph{singular} at a point $p \in D$ if $u$ is not locally bounded at $p$ (i.e. $u \neq 0 + O(1)$). In particular, $u$ is singular at $p$ if   $u(p) = -\infty$.

Some of the most important invariants of psh singularities are Lelong numbers and multiplier ideals (cf. \cite{D11}, \cite{Ki94}). The former provides a notion of `multiplicity' while the latter is concerned with integrability : the multiplier ideal $\JJ(u)$ of a psh function $u$ on a complex manifold is the ideal sheaf of holomorphic functions germs $f$ such that $\abs{f}^2 e^{-2u}$ is locally integrable. 

 These two notions of different nature are related via the following useful weakening of `equivalent singularities'.

 \begin{definition}\label{vequiv}
 
Let us define two psh functions $\vp$ and $\psi$ on a complex manifold $X$ to be valuatively equivalent, or \textbf{v-equivalent} if $v(\vp) = v(\psi)$ for every divisorial valuation $v$ on $X$ (cf. \cite{BFJ}). 
\end{definition}

Here, a divisorial valuation refer to the generic Lelong number along a smooth irreducible hypersurface $E \subset X'$ of the pullback of psh functions $\mu^* \vp, \mu^* \psi$ where $\mu: X' \to X$ is a modification. See \cite[B.5]{BBJ} for more information.  

 If $\vp = \psi + O(1)$, then they are v-equivalent. However, the converse is not true : see Proposition~\ref{siusiu} (see also \cite[(2.3), (2.9)]{KS19}). Thanks to \cite[Thm. A]{BFJ} combined with the strong openness theorem \cite{GZ}, $\vp$ and $\psi$ are v-equivalent if and only if their multiplier ideals are equal : $\JJ(c\vp) = \JJ(c\psi)$ for every real $c > 0$.

 We recall the following subclasses of psh singularities which will be used in this paper.

\begin{definition} \label{5psh}

 Let $u$ be a psh function on a complex manifold $X$. 
  We define the following conditions on $u$ :

\begin{enumerate}[I.]

\item (analytic singularities) cf. \cite{D11} We say $u$ has analytic singularities if it can be written locally as $u = c \log (\sum^N_{j =1} \abs{f_j}^{}) + O(1)$ for a real number $c \ge 0$ and local holomorphic functions $f_j$. If the functions come from a coherent ideal sheaf  $\mfa$ on $X$, i.e. $\mfa$  is  locally generated by $f_1, \ldots, f_N$,  we will use the notation 

\begin{equation*}\label{nota}
 u = c \log \abs{\mfa} + O(1)
\end{equation*}
 and say that $u$ has analytic singularities of type $\mfa^c$, cf. \cite[p.641]{BBJ}. 
\\

\item (H\"older psh)  \cite[Def. 2.3]{DK} We say $u$ is H\"older psh (or locally exponentially H\"older continuous)  if we have  $ \abs{e^{u(x)} - e^{u(y)}} \le  C d(x,y)^{\alpha} $ for   $\forall x, y \in K$ for every compact $K \subset X$ where  $C = C_K \ge 0$ and $\al = \al_K > 0$. Here $d$ is a Riemannian metric on $X$. 
\\

In order to describe the following items, we will consider germs of psh functions.  Take a neighborhood of a point in $X$ which we may identify with $D \subset \CC^n$, a bounded pseudoconvex domain containing $0 \in \CC^n$.    Let $(u_m)_{m \ge1}$ be the Demailly approximation of $u$ defined by $$ u_m (z) =  \frac{1}{m} \sup \{ \log \abs{f(z)} : f \in \OO(D), \int_D \abs{f}^2 e^{-2mu} dV <1 \} .$$

\noi See  \cite{D92}, \cite[Thm. 13.2]{D11} (also \cite[\S 2.4]{R13}) for basic properties of the Demailly approximation.  In particular, $u_m$ has analytic singularities of type $\JJ(mu)^{\frac{1}{m}}$. 
\\

\item (tame) \cite{BFJ} We say $u \in \PSH_0$ is tame if for every $m \ge 1$, we have $$u \le u_m + O(1) \le (1 - \frac{C}{m}) u + O(1)$$ for a constant $C > 0$.  (Note that the first inequality always holds as a basic property of the Demailly approximation.) 
\\

\item (asymptotically analytic) \cite{R10}, \cite{R13} We say $u \in \PSH_0$ has asymptotically analytic singularities if for every $\ep >0$, there exists a psh function $u_\ep$ with analytic singularities such that

$$ (1+\ep) u_\ep + O(1) \le u \le (1-\ep) u_\ep + O(1) .$$
\\

\item (bounded below by log)  We say $u \in \PSH_0$ is bounded below by log if  $u \ge v + O(1)$ for some $v$, a psh function with analytic singularities. In fact, in this paper, we will use this terminology only when $v$ has isolated singularities, in which case we can take $v = c \log \abs{z}$ for $c \ge 0$ where $\abs{z} = \abs{z_1} + \ldots + \abs{z_n}$ by Proposition~\ref{bbb}.

\end{enumerate}

\end{definition}

We have the implications I $\to$ II $\to$ III $\to$ IV $\to$ V :  namely, I $\to$ II from \cite{DK}, II $\to$ III due to \cite[Lem. 5.10]{BFJ}, III $\to$ IV clear from the definition, IV $\to$ V due to Proposition~\ref{bbb} below.

On the other hand, Example~\ref{KW1} does not satisfy III.  The following example satisfies II but not I in general when $\alpha_{j,k,l} > 0$ are not rational as in Example~\ref{alpha}.

\begin{example}\cite[(2.4)]{DK}
Let  $u = \max_{j} \log \left( \sum_k \prod_l \abs{f_{j,k,l}}^{\alpha_{j,k,l}} \right)$ where $f_{j,k,l}$ are holomorphic functions and $\alpha_{j,k,l} \in \RR_{>0}$ with the sets of indices $j,k,l$ being finite. Then $u$ is a psh function with $e^u$ locally H\"older continuous.

\end{example}

\medskip

\subsection{Plurisubharmonic weights}

 Let us now consider psh functions $u$  with  isolated singularities at $0 \in \CC^n$, i.e. psh functions which are  locally bounded outside $0 \in \CC^n$. We often simply call such a germ $u \in \PSH_0$ as a {\it psh weight} at $0$ (or just a {\it weight} at $0$). \footnote{Compare with \cite{BFJ} where $u$ is called a weight if, in addition, $e^u$ is continuous.}

  The collection of psh weights at $0$ will be denoted by $W_0$.   When not specified, a psh weight is always assumed as a psh weight at $0 \in \CC^n$. 
  
 For a psh function $u$, the \emph{complex Monge-Ampère operator} means the operator $u \mapsto (dd^c u)^n$ (whenever this is defined). For locally bounded $u$, it was shown by  \cite{BT76} that $(dd^c u)^k$ can be defined inductively as $(dd^c u)^k = dd^c [u (dd^c u)^{k-1}]$. Moreover, $(dd^c u)^n$ is well defined when $u$ is allowed some singularities, in particular when $u$ has isolated singularities at points (due to work of \cite{D93}, \cite{Si85}) : see e.g. \cite[Prop. 2.3]{D11}, \cite[\S 3.1]{R00}. 
 (Also see \cite{C04}, \cite{B06} for greater generality where $(dd^c )^n$ is defined. ) Hence the following definition makes sense. 
  
   \begin{definition}
 We will call the residual Monge-Ampère mass of a psh weight $u$ at $0$, $(dd^c u)^n (\{ 0 \})$  as the $n$-th Lelong number $L_n (u, 0)$.

 \end{definition} 
  
 See \cite{D93}, \cite{DH}, \cite{R11}, \cite{R13}, \cite{KR18} for more information on this and other higher Lelong numbers of psh weights.

 On the other hand, it is worth mentioning the following fact. 

\begin{proposition}\label{bbb}

Let $u$ be a psh weight at $0$, i.e. $u \in W_0$. We have $u$ bounded below by some psh weight $v$ with analytic singularities if and only if $u$ is bounded below by log, i.e. $u \ge c \log \abs{z} + O(1)$ for some $c > 0$.

\end{proposition}

\begin{proof}

Even though $u$ and $v$ are given as germs, the relation of `bounded below' is well-defined as observed in Lemma~\ref{trivial}. 
 It suffices to show one direction. Suppose that $ u \ge v + O(1)$ where $v$ is a psh weight with analytic singularities of type ${\mfa}^c$, i.e. $v = c \log \abs{\mfa} + O(1)$ for $\mfa$, a $\mfm$-primary ideal. Here we may view $\mfa$ and $\mfm$ as ideals in the analytic local ring $\OO_{\CC^n, 0}$ at $0$ since we are considering germs of holomorphic and psh functions.  Now there exists $k \ge 1$ such that ${\mfm}^k \subset \mfa$ by Nullstellensatz for the analytic local ring~\cite{GR65}. It follows that $v \ge c k \log \abs{z} + O(1)$. 
\end{proof}

\begin{lemma}\label{trivial}

Let $u$ be a psh weight at $0$.  Let $v$ be a germ of a psh function at $0$. Let $D_1 \subset \CC^n$ and $D_2 \subset \CC^n$ be open neighborhoods of $0 \in \CC^n$. Suppose that $u$ and $v$ are defined on $D_1 \cup D_2$. Then we have $v \le u + O(1)$ on $D_1$ if and only if $v \le u + O(1)$ on $D_2$.

\end{lemma}

\begin{proof}

 Note that $D_1 \cap D_2 \neq \emptyset$. The lemma is trivial since $u$ is locally bounded outside $0$. 
\end{proof}

 The following is a psh weight that does not satisfy V:

\begin{example}[J.-P. Demailly]\label{Demailly}

Consider a psh (in fact subharmonic) function $u$ on $\CC$ given  by $$u(z) = \sum^{\infty}_{k=1} u_k (z) :=   \sum^{\infty}_{k=1} 2^{-k} \log \left( \abs{z - e^{-2^k}}^2 + a_k \right) $$  where we assume that $0<a_k<1$ and the sequence $a_k$ decreases to $0$ as $k \to \infty$. The function $u(z)$ is locally bounded outside $0$ because for any $z\neq 0$, there are only finitely many $k \ge 1$ with $e^{-2^k} > \frac{1}{2} \abs{z}$ : hence all the other terms are greater than $2^{-k} \log(\frac{1}{2} \abs{z})$.

Now suppose that $u(z) \ge C\log \abs{z} + O(1)$ for some $C > 0$. For any $k_0$, we have

$$ O(1) -C 2^{k_0} \le u(e^{-2^{k_0}}) = 2^{-k_0} \log a_{k_0} + \sum_{k \neq k_0} u_k (e^{-2^{k_0}})    < 2^{-k_0} \log a_{k_0} + \log 2 $$


\noi which fails to be true if  $2^{-2k} \log a_k$ tends to $-\infty$ as $k \to \infty$.

 \end{example}

\begin{remark1}

 A psh weight bounded below by log was called as having finite Lojasiewicz exponent in \cite{R13}, \cite{R16}. It is known that a toric (i.e. multicircled) psh weight is bounded below by log (cf. \cite[p.105]{R16}, \cite{R01}).

\end{remark1}

\begin{remark1}
 For psh functions $u$ and $v$, it is customary to refer to the relation $u \ge v + O(1)$ as $u$ being `less singular' than $v$ following \cite{D11}. While this is very useful terminology, one should keep in mind that when $u \ge C \log \abs{z} + O(1)$, it will be  typically the case that the singularities of $u$ are still extremely complicated and not at all `simpler' in any way than those of $C \log \abs{z}$. In particular, $u$ need not have analytic singularities. 
\end{remark1}

\begin{remark1}

 In the algebraic aspect of the subject matter of this paper,  algebraic local rings provide an alternative setting to consider $\mfab$ and $\mfbb$ as in \cite{ELS03}, \cite{M02}. For some standard relations between algebraic and analytic local rings at a point of a complex algebraic variety, we refer to  \cite{S55}. 

\end{remark1}

\medskip

\subsection{Maximal plurisubharmonic weights}

 We will say that a psh weight  $u$ at $0 \in \CC^n$ (i.e. $u \in W_0$)  is \textbf{maximal} if it satisfies $(dd^c u)^n = 0$ outside $0$. In other words, this is when the Monge-Amp\`ere measure $(dd^c u)^n$ is equal to $\ld \delta_0$ for some $\ld \ge 0$ where $\delta_0$ is the Dirac measure at $0$.     \footnote{See \cite[Preface and \S 3.1]{Kl91}, \cite{BT76}, \cite{R98} for more on maximality of psh functions and why this condition is called maximal.}

 The collection of maximal psh weights at $0$ will be denoted by $MW_0$.   
  Example~\ref{Demailly} provides an example of a non-maximal psh weight.

   To illustrate the importance of maximality, we recall the following crucial result from \cite{R13} which was used in many results of \cite{R13} that are quoted and applied in this paper.

\begin{proposition}\cite[Lem. 2.2]{R13}\label{2.2}
 Let $D \subset \CC^n$ be a bounded hyperconvex domain containing $0$. Let $\vp, \vp_j  \; (j \ge 1) \in \PSH^{-} (D)$ be maximal psh weights that are equal to $0$ on $\partial D$.

 \begin{enumerate}
 \item

  Suppose that $\vp_j \ge \vp$ for every $j \ge 1$ and that the sequence $(\vp_j)$ decreases to a psh weight $\psi$. Then $\psi = \vp$ if and only if $(dd^c \vp_j)^n (\{ 0 \}) \to (dd^c \vp)^n (\{ 0 \})$.

 \item
 Suppose that $\vp_j \le \vp$ for every $j \ge 1$ and that the sequence $(\vp_j)$ increases to a function $\eta$. Then $\eta^* = \vp$ if and only if $(dd^c \vp_j)^n (\{ 0 \}) \to (dd^c \vp)^n (\{ 0 \})$

 \end{enumerate}

 \noi where $\eta^*$ is the upper semicontinuous regularization of $\eta$.

\end{proposition}

 The following Domination Principle plays a key role in the proof of Proposition~\ref{2.2}. We will also use it later in this paper.

 \begin{lemma}\label{DP} \cite[Lem. 2.1]{R13}, cf.  \cite[Lem. 6.3]{R06},  \cite{ACCH}. 
 
 Let $D \subset \CC^n$ be a bounded hyperconvex domain containing $0$. Let $u$ and $v$ be two maximal psh weights on $D$, equal to zero on $\partial D$. Suppose that $(dd^c u)^n (0) = (dd^c v)^n (0)$. If $u \ge v$ on $D$, then $u \equiv v$.

 \end{lemma}

 Now using the notion of maximal psh weights, we recall another invariant of  psh singularities from \cite{R06}. Let $\vp$ be a maximal psh weight in $MW_0$. The \emph{relative type} of $u$ with respect to $\vp$ is defined by  

$$\sigma (u, \vp) = \liminf_{z \to 0} \frac{u(z)}{\vp(z)} .$$

\noi This generalizes  the usual Lelong number, cf. \cite[\S 3]{R06}. It is also known to have the characterization $\sigma (u, \vp) = \sup \{ c> 0 : u \le c \vp + O(1) \}$, cf. \cite[p.1222]{R13}, \cite{BFJ}.  This notion of relative type will be used in the next subsection.

\medskip

\subsection{Green function of a plurisubharmonic weight}

 Now we recall the notion of Green functions from \cite[\S 2.3]{R13}. Let $D \subset \CC^n$ be a bounded hyperconvex domain containing $0 \in \CC^n$. Let $u \in W_0$ be a psh weight defined on $D$ (i.e. having isolated singularities at $0$). 
 
 \begin{definition}\label{Gu}
 Define the Green function $G_u  (= G_{u,D})$ of $u$ by

 \begin{equation}\label{greeni}
  G_u (z) = G_{u, D} (z) = \limsup_{y \to z} \left( g(y) :=  \sup_h \{ h(y) : h \in \PSH^{-} (D), h \le u + O(1) \} \right)
 \end{equation}
\end{definition}

\noi following \cite[\S 5.2]{R06}, \cite[p.1222]{R13} (where this was called the complete greenification). 
There it was shown that $G_u$ is maximal on $D \setminus \{ 0 \}$ and that $G_u$ is equal to zero on $\partial D$, i.e. $G_u (z) \to 0$ as $z \to \partial D$, cf. \cite[Prop. 5.6]{R06}.

\begin{remark1}\label{weak}
We remark that even though $ G_{u, D}$ does depend on the choice of $D$, dependence on a particular choice of $D$ is rather weak and not important for the purpose of studying psh singularities since $G_{u, D_1}$ and $G_{u, D_2}$ will differ by $O(1)$ on $D_1 \cap D_2 \ni 0$.
\end{remark1}

 We note the following basic properties of Green functions.

\begin{prp}\label{basic}

Let $u$ and $v$ be psh weights on a bounded hyperconvex domain $D \subset \CC^n$.

\begin{enumerate}

\item We have $u \le G_u + O(1)$.

\item  If $u \le v + O(1)$, then $G_u \le G_v $.   If $u = v + O(1)$, then $G_u = G_v$.

\item
 For $c \in \RR_{>0}$, we have $G_{cu} = c G_u$.

\item

 If $u$ has analytic singularities, then $u = G_u + O(1)$.

\end{enumerate}

\end{prp}

\begin{proof}

 Only (4) is not immediate from the definition of the Green function. This is already known from \cite[Prop. 5.1]{RS} but we present an alternative argument here.  From \cite[Thm. 4.3]{K14}, we know that if psh functions $u$ and $v$ are v-equivalent to each other and $u$ has analytic singularities, then $u \ge v + O(1)$. We can apply this for $v = G_u$ due to Proposition~\ref{vgreen}. Combining with (1), we obtain (4).
\end{proof}

In the following propositions, we recall that some important invariants of psh singularities are preserved when one takes the Green function of a psh weight.

\begin{proposition}\label{R06}

Let $u \in W_0$ be a psh weight and $G_u$ its Green function on a bounded hyperconvex domain $0 \in D \subset \CC^n$. 

(1)  We have $L_n (u, 0) = L_n (G_u, 0)$.
 
 (2) The relative types $\sigma (u, \vp) = \sigma (G_u, \vp)$  with respect to all maximal weights $\vp$. 
  
 \end{proposition}

 This result was given in \cite{R06}. 
(1)  is  due to \cite[Prop. 5.6]{R06}. (See also  \cite[Prop. 2.1]{R11} for generalization to other higher Lelong numbers.) 
(2)  is due to  \cite[Prop. 5.5]{R06} (cf. \cite[p.1222]{R13}).

\begin{proposition}\label{vgreen}

Let $u \in W_0$ be a psh weight and $G_u$ its Green function on a bounded hyperconvex domain $0 \in D \subset \CC^n$. Then $u$ and  $G_u$ are v-equivalent, i.e. the multiplier ideal sheaves are equal $\JJ(mu) = \JJ(m G_u)$ for all real $m > 0$.

\end{proposition}

\begin{proof}

As in \cite{BFJ}, \cite{GZ}, first define  the `limit multiplier ideal' $\JJ_+ (u)$ to be $\cup_{\ep >0} \JJ((1+\ep)u))$ which is a coherent ideal sheaf due to strong Noetherian property.

From Proposition~\ref{R06},  we have  $\sigma (u, \vp) = \sigma (G_u, \vp)$  with respect to all maximal psh  weights $\vp \in MW_0$.  This implies that, thanks to  \cite[Thm. A]{BFJ},  we have $\JJ_+ (mu) = \JJ_+ (m G_u)$ for every real $m > 0$. 
By the strong openness theorem \cite{GZ}, we have $\JJ (v)= \JJ_+ (v)$ for every psh $v$, hence now we get $\JJ (mu) = \JJ (m G_u)$. 
 \end{proof}

\medskip

\section{Plurisubharmonic functions associated\\ to a graded system of ideals}\label{SG}

  In this section, we will introduce and use two kinds of psh functions one can naturally associate to a given graded system of ideals.

\subsection{Siu functions of a graded system of ideals}

 Let $X$ be a connected complex manifold. A \emph{graded system of ideals} in $\OO_X$  is a sequence of coherent ideal sheaves  $\mfab = (\mfa_m)_{m \ge 1}$ satisfying $\mfa_m \cdot \mfa_k \subset \mfa_{m+k}$ for all $m, k \ge 1$. First we recall the following

 \begin{definition}\label{am}\cite[(11.1.15)]{L} cf. \cite{ELS01}
 The asymptotic multiplier ideal sheaf of $\mfab$ with real  coefficient $c > 0 $ (denoted by $\JJ(c \cdot \mfab)$ or by $\mfb_c$) is defined as the unique maximal member in the family of ideal sheaves $\{ \JJ (\frac{c}{q} \cdot \mfa_q) \}_{q \ge 1}$.
 \end{definition}

 Here the algebraic multiplier ideal $\JJ (\frac{c}{q} \cdot \mfa_q) $ can be understood as the multiplier ideal of a psh function, namely as  $\JJ(\frac{c}{q} \log \abs{\mfa_q})$. We will  see that the asymptotic multiplier ideal can be also understood as the multiplier ideal of a psh function. 
 
 Given $\mfab$, we define a Siu function associated to $\mfab$ following \cite{S98} \footnote{See also \cite{BEGZ}, \cite{K14} for more information.} by

 \begin{equation}\label{siu}
  \vp = \vp_{\mfab} = \log \left( \sum^{\infty}_{k \ge 1}  \ep_k  \abs{\mfa_k}^{\frac{1}{k}} \right)
 \end{equation}

\noi on a domain $D \subset X$ where every graded piece $\mfa_k$ is an ideal with a choice of a finite number of generators, say $g_1^{(k)}, \ldots, g_{m_k}^{(k)}$. 
Here  we used  the notation $ \abs{\mfa_k} := \abs{g_1^{(k)}} + \ldots + \abs{g_{m_k}^{(k)}}$ as in Definition~\ref{5psh} (I),   with the convention that whenever the notation  $ \abs{\mfa_k}$ is used, a specific choice of a finite number of generators is implicitly assumed.
Also $ \ep_k $'s are a choice of positive coefficients such that the series $\sum \ep_k$ converges.

Given $\mfab$, one can always take a domain  $D$ for $\vp_{\mfab}$ as a relatively compact Stein domain due to coherence of each $\mfa_k$. Namely, $\mfa_k$ is generated by $\mfa_k (D)$ due to Cartan's Theorem A~\cite[p.243]{GR65}  and then the strong Noetherian property of coherent sheaves~\cite[II (3.22)]{DX} applies.  For the purpose of this paper, bounded hyperconvexity of a domain in $\CC^n$ is enough to assume when defining both Siu and Green functions, thanks to the following well-known fact.

\begin{lemma}\label{hyper}

 If $D \subset \CC^n$ is a bounded hyperconvex domain, then it is Stein. 

\end{lemma}

\begin{proof}

 By definition, $D$ admits a negative psh exhaustion function which implies that $D$ is pseudoconvex \cite[I (7.2)]{DX} and thus Stein by \cite[VIII (9.11)]{DX}. 
\end{proof} 

\begin{remark1}
 The converse of this statement is known under some conditions: for example, a bounded pseudoconvex domain $D \subset \CC^n$ with Lipschitz boundary is hyperconvex due to \cite[(0.2)]{D87}. 
\end{remark1}

 We recall the following important fact due to S. Boucksom which directly relates a Siu function associated to $\mfab$ to   the asymptotic multiplier ideals (Definition~\ref{am}).

\begin{theorem}[S. Boucksom]\cite[Thm. 2.2]{KS19}\label{ksb}
 Let $\mfab$ and $\vp = \vp_{\mfab}$ be as above. For every real $c > 0$, we have  $\JJ(c\vp) = \JJ  (c \cdot \mfab)$.

\end{theorem}

 In particular, thanks to this result, the formal notation for the asymptotic multiplier ideals $ \JJ  (c \cdot \mfab)$ (as in \cite[Ch.11]{L}) using the coefficient $c$ can be naturally understood as the multiplier ideal $\JJ(c\vp)$ associated to the psh function $c \vp$.


\medskip

\subsection{Green function of a graded system of ideals}

Now we will recall from \cite{R13} the notion of the Green function associated to a graded system of ideals $\mfab$. Let $\mfab$ be a graded system of ideal sheaves on a complex manifold $X$. Let $D \subset X$ be a domain which we can view as a bounded hyperconvex neighborhood of $0 \in \CC^n$. 

First, for an individual ideal $\mfa$, we will define $G_\mfa$, the Green function with singularities along $\mfa$ in the sense of Rashkovskii-Sigurdsson~\cite{RS}.

 Let $\mathcal F_{\mfa}$ be the class of psh functions $u \in \PSH^{-} (D)$ satisfying $u \le \log \abs{ \mfa} + O(1)$ on $D$. 
We define $G_\mfa (z) = \sup \{ u(z) : u \in \mathcal F_{\mfa} \}$ for $z \in D$. From \cite[Thm. 2.5]{RS}, $G_\mfa$ is psh.  When $\mfa$ is $\mfm$-primary, $G_\mfa$ is equal to $G_{\log \abs{\mfa} }$, the Green function of a psh weight $\log \abs{\mfa}$, cf. Definition~\ref{Gu}. 
 (Note that $G_{\log \abs{\mfa} }$ is well defined, independent of the actual psh function represented by the notation $\log \abs{\mfa}$.)

Now given a graded system of ideals $\mfab$, let $h_k := \frac{1}{k} G_{\mfa_k}$.
Since $\mfa_k^m \subset \mfa_{km}$ holds from the defining condition for a graded system of ideals, we have the relation $$ m \log \abs{\mfa_k} \le \log \abs{\mfa_{km}} + O(1).$$  Hence we have  $h_k \le h_{km}$. In particular, $h_{k!}$ is an  increasing sequence in $k$ and it converges to 
 $h_{\mfab}$ defined by $h_{\mfab} (z) := \sup_m h_m (z)$.

\begin{definition}\label{Gadot} \cite[\S 5]{R13}
 We define  the {Green function} $G_{\mfab}$ associated to $\mfab$ to be the upper semicontinuous regularization $G_{\mfab}:= (h_{\mfab})^*$.

\end{definition}

  We now have the following properties.

\begin{lemma}\label{GG}

Let $G_{\mfab}$ and $\vp_{\mfab}$ be Green and Siu functions defined as above on $D$ for  a graded system of ideals $\mfa_\bullet$.

(1) For every real $c > 0$, we have the equality of  multiplier ideals  $$ \JJ(c G_{\mfa_\bullet} ) = \JJ  (c \cdot \mfab) = \JJ(c \vp_{\mfa_\bullet}).$$ In particular, the psh functions $G_{\mfab}$ and $\vp_{\mfab}$ are v-equivalent.

(2)  We have  $G_{\vp_{\mfa_\bullet}} \ge  G_{\mfa_\bullet} $ on $D$.

\end{lemma}

 \begin{proof}

(1) 
 Let $h_k=\frac1{k} G_{\mfa_{k}}$ (as above). Then the subsequence $h_{k!}$ increases almost everywhere to $G_{\mfab}$.  
  By the version of the strong openness  theorem~\cite{GZ} for increasing sequences (see, e.g., \cite{Hi17}, \cite{Le17}, cf.\cite{Hi}), $\JJ(h_{k!})$ stabilizes  for $k$ sufficiently large and is equal to $\JJ(G_{\mfab})$.  By definition of the asymptotic multiplier ideal, the stabilized $\JJ(h_{k!})$ is equal to the asymptotic multiplier ideal.

  The second equality is precisely Theorem~\ref{ksb}. We remark that the same argument as in the first equality using the strong openness with respect to the increasing sequence given by partial sums in the infinite series defining $\vp_{\mfab}$ provides an alternative proof for Theorem~\ref{ksb}. \footnote{We note that both arguments depend on the strong openness theorem~\cite{GZ}, cf.\cite{Le17}, \cite{Hi17}. }

 (2)  Let $\vp := \vp_{\mfab}$.   Since $\vp\ge \frac1k\log|\mfa_k|+ \log\epsilon_k$ from \eqref{siu}, we have $$G_\vp \ge G_{\frac1k\log|\mfa_k|}=\frac1kG_{\log|\mfa_k|}=h_k .$$ Taking $k=m!$ with $m\to\infty$, this gives $G_\vp\ge G_{\mfab}$ since $h_{m!}$ converges increasingly to $h_{\mfab}$ whose upper semicontinuous regularization $(h_{\mfab})^*$ is equal to $G_{\mfab}$.

 \end{proof}

 Note that the Green function $G_{\mfab}$ is uniquely determined by a graded system of ideals $\mfab$ (on a given domain) while the Siu function $\vp = \vp_{\mfab}$ depends on choices of generators and coefficients taken in its definition \eqref{siu}. We have the following fact. 
 
\begin{prop}\label{siusiu}

\begin{enumerate}

\item There exist a graded system of ideals $\mfab$ (on some $X$) and its Siu functions $\vp:=\vp_{\mfab}$ and $\psi :=\psi_{\mfab}$ such that  $\vp$ and $\psi$ do not have equivalent singularities.

\item  Given a graded system of ideals $\mfab$, different Siu functions of $\mfab$ are always v-equivalent. 

\end{enumerate}

\begin{proof} 

For (1), in \cite[Thm.3.5]{K14}, such an example of $\mfab$ was given using a section ring of a line bundle that is not finitely generated on a smooth projective variety. 

 For (2), let $\vp := \vp_{\mfab}$ be an arbitrary Siu psh function associated to given $\mfab$.  We note that for every real $c>0$, the multiplier ideal $\JJ (c \vp)$ is equal to the asymptotic multiplier ideal $\JJ( c \cdot \mfab)$ by Lemma~\ref{GG} (1). As mentioned after Definition~\ref{vequiv}, this means that the v-equivalence class of different Siu psh functions of $\mfab$  is uniquely determined. 
\end{proof}

\end{prop}

 \medskip

\subsection{Zero-dimensional ideals}

Now we specialize to the case when a graded system of ideals $\mfab$ consists of zero-dimensional (i.e. $\mfm$-primary) ideals at a point of a complex manifold.  In this case,  we have the relation $G_{\mfa_1} \le G_{\mfab}$ from the construction of $G_{\mfab}$. Then it follows that $G_{\mfab}$ is maximal on $D \setminus \{ 0 \}$ and that $G_{\mfab}$ is equal to zero on $\partial D$, i.e. $G_{\mfab} (z) \to 0$ as $z \to \partial D$ since these properties are satisfied by $G_{\mfa_1} = G_{\log \abs{\mfa_1}}$ (cf. \S 2.4). 

\begin{proposition}\label{eaga}

 If $\mfa_k$ is zero-dimensional  for every $k$,  we have the equality:
 $$e(\mfab) = L_n (G_{{\mfab}}, 0).$$
 

\end{proposition} 

\begin{proof} 
 This was given in the proof of \cite[Prop. 5.1]{R13} :  we  recall the argument for the convenience of readers. In the definition of $G_{\mfab}$, we used the sequence $h_k = \frac{1}{k} G_{\mfa_k}$. 
 By \cite[Lem. 2.1]{D09}, for each $k$, we have  $L_n (h_k, 0) = \frac{1}{k^n} e(\mfa_k)$.  Since the subsequence  $h_{k!}$ is increasing and converges to $G_{\mfab}$ almost everywhere, we have $L_n (h_{k!}, 0) \to L_n (G_{\mfab}, 0)$.   This should be equal to the limit of $\frac{1}{k^n} e(\mfa_k)$, which is nothing but $e(\mfab)$ (by the definition of $e(\mfab)$). 
\end{proof} 

Now we prove Theorem~\ref{eaebeq}. 

\begin{proof}[Proof of Theorem~\ref{eaebeq}]
  
    Let $G := G_{\mfab}$. We know from \cite{D92}, \cite{D13} that the $m$-th Demailly approximant of $G$ satisfies $$G_m = \frac{1}{m} \log \abs{\JJ(mG)} + O(1) .$$ Since $G_m$ has analytic singularities, its residual Monge-Amp\`ere mass is equal to the Samuel multiplicity by \cite[Lem. 2.1]{D09} :

 \begin{equation}\label{LG}
  \frac{e( \JJ(m G) ) }{m^n}  = L_n ( G_m , 0) .
  \end{equation}

This yields the last equality of \eqref{GGG} below. 
Since it is well known from \cite{ELS03} that $e(\mfa_\bullet) \ge e(\mfb_\bullet)$ always holds, we  have

 \begin{equation}\label{GGG}
   L_n (G, 0) = e(\mfa_\bullet)   \ge e(\mfb_\bullet) = \lim_{m \to \infty}  \frac{e(\mfb_m)}{m^n}    =  \lim_{m \to \infty}  \frac{e(\JJ(m G)) }{m^n} = \lim_{m \to \infty}  L_n (G_m, 0)
 \end{equation}
  where the second to last equality follow from  Lemma~\ref{GG} (1).  The first equality is given by Proposition~\ref{eaga}.  The conclusion now follows from \eqref{GGG} since Demailly's strong continuity for $G$ states that $L_n (G, 0) = \lim_{m \to \infty}  L_n (G_m, 0)$. 
\end{proof}

 Now combining this with Theorem~\ref{eaeb}, we have

\begin{corollary}\label{Conjec}

Demailly's strong continuity holds for the psh functions $G_{\mfa_\bullet}$ and $\vp_{\mfab}$ associated to a graded system of ideals $\mfab$.

\end{corollary}

\begin{proof}[Proof of Corollary~\ref{Conjec}]

For $G_{\mfa_\bullet}$, it follows from Theorem~\ref{eaebeq} and Theorem~\ref{eaeb}.
For $\vp = \vp_{\mfab}$,  we see that

  \begin{align*}
  L_n (\vp, 0)   &= L_n (G_\vp , 0) \qa \text{  \qa \qa \qa  by Proposition~\ref{R06} } \\
  &\le  L_n (G_{\mfa_\bullet}, 0) \text{   \qa \qa \qa \;   by Lemma~\ref{GG} (2) and Comparison Theorem~\cite{D93}  } \\
  & = e(\mfab)  \text{ \qa \qa \qa \qa \qa \qa   by Proposition~\ref{eaga} } \\
  &= e(\mfbb) \qa \qa  \qa \qa  \text{ \qa \qa  by Theorem~\ref{eaeb} }   \\
  &  = \lim_{m \to \infty} \frac{e(\mfb_m)}{m^n} = \lim_{m \to \infty} L_n (\vp_m, 0) \;   \le L_n (\vp, 0)
  \end{align*}

 \noi which should be a chain of equalities. Here $\vp_m$ is the $m$-th Demailly approximation of $\vp$.

 The equality $\displaystyle \frac{1}{m^n} e(\mfb_m) =  L_n (\vp_m, 0)$ holds due to  the fact that for a psh function with analytic singularities given by a $\mfm$-primary ideal, the $n$-th Lelong number is equal to the Samuel multiplicity of the ideal~\cite[Lem. 2.1]{D09}. The limit $\lim_{m \to \infty} L_n (\vp_m, 0)$ exists since the previous limit defining $e(\mfbb)$ exists. The relation $\vp \le \vp_m + O(1)$ was used in the last inequality.
\end{proof}

\begin{remark1}
In the special case when $\mfab$ consists of monomial ideals, Corollary~\ref{Conjec} was known by \cite{R13} where Demailly's strong continuity was shown for every toric psh weights (cf. \cite{KR18} for generalization of this to toric psh functions in the Cegrell class).
 \end{remark1}

 Another consequence of Theorem~\ref{eaeb} is the following

\begin{corollary}\label{ggvp}

 Let $\mfab$ be a graded system of $\mfm$-primary ideals at a point $p$ (with the maximal ideal $\mfm$) of a complex manifold $X$.  Let $\vp = \vp_{\mfab}$ be a Siu function of
$\mfab$ defined in a bounded hyperconvex domain $D$ in $X$ containing  $p$. Then we have the equality $G_{\vp} \equiv G_{\mfab}$ on $D$. 

\end{corollary}

\begin{proof}

 Note that we have the inequality $G_\vp \ge   G_{\mfa_\bullet} $ from Lemma~\ref{GG} (2). Also we have the equality   $L_n (G_\vp , 0) =  L_n (G_{\mfa_\bullet}, 0)$ from the proof of Corollary~\ref{Conjec}, which used Theorem~\ref{eaeb}. Now we apply the Domination Principle, Lemma~\ref{DP} to conclude. (Note that the conditions needed in Lemma~\ref{DP} for $G_{\mfab}$ (as well as for $G_{\vp}$) are already mentioned in the beginning of this subsection.) 
\end{proof}

 Recall from Proposition~\ref{siusiu} that, for general $\mfab$, the equivalence class of Siu psh functions is not necessarily uniquely determined. 
 Namely, there exists $\mfab$ (cf. \cite[Thm. 3.5]{K14}) for which there can be infinitely many different  Siu functions (say $\vp, \psi, \ldots$) with mutually non-equivalent singularities.  Note however they are all v-equivalent to each other by Proposition~\ref{siusiu}. 
 Corollary~\ref{ggvp} says that these v-equivalent psh functions  share the same Green function ($G_\vp = G_\psi = \ldots$) since $G_{\mfab}$ depends only on $\mfab$.

 We also remark that, in this regard,  we do not know of an example of psh weights $u$ and $v$ that are v-equivalent to each other but having different Green functions $G_u \neq G_v$ (on a bounded hyperconvex domain).

 \begin{example}\label{ggaa}

 In the case of toric psh functions and monomial ideals, we can say more.  As in \cite[(2.8), (2.9)]{KS19}, let $\mfab$ and $\mfab'$ be two different graded systems of monomial ideals which share the same Newton convex body. Then Siu functions of $\mfab$ and of $\mfab'$ are all v-equivalent to each other. Their Green functions $G_{\mfab}$ and $G_{\mfab'}$ also coincide since each of them is equivalent to the indicator function which is equal to the Green function for the the unit polydisk, cf. \cite[Thm. 3.1]{R11}.

 \end{example}

\medskip

\section{K\"uronya-Wolfe constants of graded systems of ideals}

 K\"uronya and Wolfe considered the notion of stability of graded systems of ideals. Roughly speaking, $\mfab$ is stable if it has strictly positive asymptotic vanishing order along every irreducible subvariety $Z$ that appears in the support of $\mfa_m$ for $m \gg 0$. In the case of $\mfm$-primary ideals, since the only possible $Z$ is the point at hand,  it simplifies as follows.

 \begin{definition}\label{stable}\cite{KW}
 Let $X$ be an irreducible smooth complex variety. Let $\mfab$ be a graded system of $\mfm$-primary ideals where $\mfm$ is the maximal ideal of a point $p \in X$. We say that $\mfab$ is stable if

(1) $\mfa_m \neq 0$ for all $m \gg 0$ and

(2) $\ord_p (\mfab) > 0$.
 \end{definition}

 Here $\ord_p (\mfab)$ is the asymptotic order of vanishing of $\mfab$ at $p$, cf. \cite[Def. 2.2]{KW}. The condition (1) is in fact  built into the definition of  $\mfm$-primary graded systems of ideals we consider in this paper (see \S 1. Introduction). \footnote{ In general when without the condition (1), one can consider the subsemigroup of indices $S = S(\mfab) := \{ m \ge 1: \mfa_m \neq (0) \}$ as in \cite[\S 2.1]{JM}.} Also the condition (2) is not restrictive (to assume in our main result of this section, Corollary~\ref{KWC}) either, since the special  case $\ord_p (\mfab)=0$  is already known to be equivalent to $e(\mfab)=0$ from \cite[Thm. 1.7, Remark 3.8]{M02}.

 Now assuming Definition~\ref{stable}, the main result of \cite{KW}, Theorem 3.3,  showed the following  (here we restrict to the case of zero-dimensional ideals) :

 \begin{theorem}[K\"uronya-Wolfe]\label{KW}

  Let $\mfab$ be a stable graded system of $\mfm$-primary ideals. Then there exist positive real constants $C$ and $D$ such that for all sufficiently large integer $m$, we have the following containment for the asymptotic multiplier ideals (Definition~\ref{am})

 \begin{equation}\label{Kuro}
  \mfb_{Cm + D} = \JJ \left( (Cm + D) \cdot \mfab \right)  \subseteq  \mfa_m .
 \end{equation}

 \end{theorem}

 The first equality in \eqref{Kuro} simply refers to the two different notations. 

\begin{definition}\label{kwc}
 When \eqref{Kuro} holds, we will say that $C > 0$ is a K\"uronya-Wolfe constant for $\mfab$.
\end{definition}

 Note that we are not taking the infimum of such $C$'s in this definition since the infimum may not be a K\"uronya-Wolfe constant.

 \begin{remark1}
   In \cite[Thm. 3.3]{KW},  $ \JJ \left( \lceil Cm + D \rceil \cdot \mfab \right) $ is taken in the place of $ \JJ \left( (Cm + D) \cdot \mfab \right)$ to deal with only integer indices for the asymptotic multiplier ideals. It is easy to check the above statement of Theorem~\ref{KW} from its proof.
   \end{remark1}

 Now we can give an analytic proof of $e(\mfa_\bullet) = e(\mfb_\bullet)$ in the case when one can take the constant $C$ to be at most $1$ in Theorem~\ref{KW}. According to \cite[p.802]{KW}, this is indeed the case for many situations of geometric interest although it does not always hold.

\begin{corollary}\label{KWC}

 Suppose that a graded system $\mfab$ of $\mathfrak m$-primary ideals  has a K\"uronya-Wolfe constant $C \le 1$ in Theorem~\ref{KW} and satisfies $\ord_p (\mfab) >0$. Then the Green function $G_{\mfa_\bullet}$ and Siu functions $\vp_{\mfab}$ associated to $\mfab$ are tame psh weights. Therefore  $e(\mfa_\bullet) = e(\mfb_\bullet)$ holds with an analytic proof independent of Theorem~\ref{eaeb}.

\end{corollary}

\begin{proof}[Proof of Corollary~\ref{KWC}]

  Since $C \le 1$, we have  $\mfb_{m+D} \subseteq \mfb_{Cm + D} \subseteq \mfa_m$ for $m \gg 0$. Therefore it follows that $$ \frac{1}{k} \log \abs{\mfb_k}  \le  \frac{1}{k}  \log \abs{\mfa_{k-D}} + O(1) = \frac{k-D}{k}  \frac{1}{k-D} \log \abs{\mfa_{k-D}} + O(1)  $$ for $k = m+D \gg 0$. On the other hand, by the definition of  $G_{\mfa_\bullet}$ and $\vp_{\mfab}$, we have $  \frac{1}{k-D} \log \abs{\mfa_{k-D}}  \le  \vp  + O(1) $ where $\vp$ denotes $G_{\mfa_\bullet}$ or $\vp_{\mfab}$.

 Since the $k$-th Demailly approximant $\vp_k$ of $\vp$ is equal to $\frac{1}{k} \log \abs{\mfb_k}$ up to $O(1)$,  we get  $ \vp_k   \le  (1 - \frac{D}{k} ) \vp + O(1) $ which says that $\vp$ is tame when combined with $\vp \le \vp_k + O(1)$, a basic property of Demailly approximation. Since Demailly's strong continuity holds for tame  psh weights due to \cite{BFJ}, we conclude by applying Theorem~\ref{eaebeq}.
\end{proof}

\begin{remark1}\label{els421}

   Corollary~\ref{KWC} generalizes \cite[Prop. 3.11]{ELS03} since, as noted in \cite[p.421]{ELS03}, the condition therein $\delta \mfb_m \subset \mfa_m \; (m \gg 0)$  for a fixed $\delta$ implies the condition $\mfb_{m+D} \subset \mfa_m \; (m \gg 0)$ using the fact that $\mfab$ are valuation ideals.

\end{remark1}

 When all the K\"uronya-Wolfe constants are greater than (and bounded away from) $1$, we have an example of a toric psh function with isolated singularities that is not tame.

 \begin{example} cf. \cite[Example 3.6]{KW}\label{KW1}
  Let $\mathfrak{m} = (x,y)$ be the maximal ideal of the origin in $\CC^2$. Let $\mfa_k = \mathfrak{m}^k (x^k, y)$. Let the asymptotic multiplier ideals of $\mfab$ be denoted by $\mfbb$.  Since $\mfbb$ is equal to the asymptotic multiplier ideals of another graded system of ideals $\mfab' = \mathfrak{m}^k$ as was noted in \cite{KW}, it follows that $\mfb_k = \mathfrak{m}^{k + 1 - n} = \mfm^{k-1}$ for $k \ge 1$.

  Now we will show that any Siu function $\vp = \vp_{\mfab}$ (associated to $\mfab$) is not tame. Let $\vp_k \; (k \ge 1)$ be the Demailly approximation sequence of $\vp$. Suppose that there exists $C>0$ such that   $\vp + O(1) \le  \vp_k \le  (1 - \frac{C}{k} ) \vp + O(1) $ for all $k \gg 0$. Since we need this condition for $k \gg 0$, we may assume that $C \ge n = 2$ (and $ 1 - \frac{C}{k} > 0$).   Up to equivalence of singularities, we can write $\vp_k = \frac{1}{k} \log \abs{ \mfb_k }$ (again in the notation of Definition~\ref{5psh}, (I)) and $\vp = \log (\sum_{p \ge 1}  \ep_p   \abs{\mfa_p}^{\frac{1}{p}}   )   $.   We then have
  \begin{align*}
 \frac{1}{k}   \log   \abs{\mathfrak{m}}^{k+1-n }   &\le    (1 - \frac{C}{k}) \log (\sum_{p \ge 1}  \ep_p   \abs{\mfa_p}^{\frac{1}{p}}   ) + O(1) \\
 &= (1 - \frac{C}{k}) \left( \log \abs{\mathfrak{m}} + \log (\sum_{p \ge 1}  \ep_p   \abs{(x^p, y)}^{\frac{1}{p}}  ) \right)  + O(1)
 \end{align*}
\noi  since  $\mfa_p =  \mathfrak{m}^p  (x^p, y) $. Hence we obtain, for $k \gg 0$,

 \begin{align*}
   (C+1-n) \log \abs{\mathfrak{m}} &\le  (k-C)   \log (\sum_{p \ge 1}  \ep_p   \abs{(x^p, y)}^{\frac{1}{p}}  )   + O(1) \text{\; and } \\
  (C+1-n) \log  \abs{x}    &\le  (k-C)   \log \abs{x}    +O(1)
  \end{align*}

\noi  where the last inequality is obtained when we restrict the psh functions on both sides to the line $y=0$, which gives contradiction for $k \gg 0$. Thus $\vp$ is not tame.

 \end{example}

\begin{remark1}

 As mentioned in \cite[Example 3.8]{R13}, a toric maximal psh weight is tame. This implies that $\vp = \vp_{\mfab}$ in this example is not maximal, even up to $O(1)$. Also this provides an example of $\vp$ and $\psi$ such that they are v-equivalent but only one of them is tame : take $\psi$ to be the Green function of $\vp$ which is nothing but $G_{\mfab}$ by Corollary~\ref{ggvp}.

\end{remark1}

\medskip

\section{Sup-analytic singularities}

 In this section, we will generalize Theorem~\ref{gs} to psh weights with \emph{sup-analytic singularities}, a natural class of psh weights which was introduced by \cite{R13}. They are defined as (up to greenification) the increasing limits of maximal psh weights with analytic singularities (Definition~\ref{supa}).   We will see that this class of psh weights with sup-analytic singularities indeed includes Green and Siu functions in Theorem~\ref{gs}, cf. Corollary~\ref{greensiu}.

 Throughout this section, we will consider psh weights $\vp$ on $D \subset \CC^n$, a bounded hyperconvex neighborhood of $0 \in \CC^n$. The $n$-th Lelong number $L_n (\vp, 0)$ is often denoted simply by $L_n (\vp)$. Also recall that the collection of all maximal psh weights is denoted by $MW_0$. 
 
\subsection{Sup-analytic singularities}

 Let $\vp$ be a psh weight at $0 \in \CC^n$, i.e. a psh germ with isolated singularities at $0$. We first recall a natural way to associate a graded system of ideals $\mfab$ to $\vp$. Define

 \begin{equation}\label{gdd}
 \mfa_k = \mfa_k (\vp) = \{ f \in \OO_{\CC^n, 0} :  \log \abs{f} \le k \vp + O(1) \} .
\end{equation}

\noi Here $\mfa_k$ is well-defined in view of Lemma~\ref{trivial}.  \footnote{Since a zero-dimensional ideal is determined by its stalk at $0$, it is indeed sufficient to require the condition in \eqref{gdd} for a germ $f$. }
 It is clear that $\mfa_k \mfa_l \subset \mfa_{k+l}$. In general, it is possible that all the graded pieces $\mfa_k$ are consisting of zero elements only. On the other hand, it is easy to see that $\mfab$ defined this way is a graded system of $\mfm$-primary ideals with $\mfa_k \neq \{ 0 \}$ for every $k$, exactly when $\vp$ is bounded below by log (see Definition~\ref{5psh}) using Proposition~\ref{bbb}.

\begin{remark1}
 When $\vp$ is a maximal psh weight, it is also possible to define $\mfa_k$ in \eqref{gdd} using relative types as in \cite[p.1231, (5.8)]{R13}.
\end{remark1}

 We recall the following definition from \cite{R13}, in which case $\mfab(\vp)$ defined by \eqref{gdd} will turn out to behave well.

\begin{dfn}\label{supa}

We will say that a maximal psh weight $\vp\in MW_0$ has {\textbf{sup-analytic singularities}} if  there exists a domain $D$ and  an increasing sequence of maximal psh weights $\psi_j\in \PSH(D)$ with analytic singularities such that  $\psi_j \to {\vp}$ almost everywhere on $D$ as $j \to \infty$.

Also we say that a psh weight $\vp \in W_0$ has sup-analytic singularities if its Green function $G_\vp$ on some domain $D$ does so.

\end{dfn}

Note that if a maximal psh weight has sup-analytic singularities, then it is bounded below by log, due to Proposition~\ref{bbb}.

\begin{example}\label{alpha}

 Let $\vp = \log (\abs{x} + \abs{y}^{\al} )$ near the origin of $\CC^2$ where $\al > 0$ is irrational. By \cite[Example 4.1]{K15}, $u$ does not have analytic singularities, but it is exponentially H\"older continuous. On the other hand, take $\psi_j = \log (\abs{x} + \abs{y}^{\al_j} )$ where $\al_j  > 0$ is a sequence of rational numbers converging to $\al$.  Then $\vp$ and $\psi_j$  are maximal psh weights  by King's formula~\cite[Chap. III, (8.18)]{DX} : thus Definition~\ref{supa} is checked for $\vp$.

\end{example}

 Sup-analytic singularities have the following characterization.

\begin{proposition}\label{sup-an} {\rm \cite[Prop. 5.4, 5.5]{R13}}
 Let $\vp$ be a maximal psh weight, i.e.  $\vp \in MW_0$. The following are equivalent :

 \begin{enumerate}

 \item

 $\vp$ has sup-analytic singularities.

 \item

 There exists a sequence of psh functions $\psi_j\in \PSH(D)$ with analytic singularities such that $\psi_j \le \vp$ and $L_n(\psi_j)\to L_n(\vp)$ as $j \to \infty$.

 \item

 We have $G_{\mfab (\vp)} \equiv G_\vp$ for the graded system of ideals $\mfab (\vp)$ defined by \eqref{gdd}.

 \end{enumerate}

\end{proposition}

  We remark on the last condition:  it is clear that $G_{\mfab(\vp)}\le G_\vp$. Furthermore, the two functions coincide if and only if $L_n(G_{\mfab(\vp)})= L_n(G_\vp)$ by Lemma~\ref{DP}.

Now we have the following three corollaries of Proposition~\ref{sup-an}. 

\begin{cor}\label{greensiu}

Given a graded system of $\mfm$-primary ideals $\mfab$, the Green function $G_{\mfab}$ and Siu functions
 $\psi_{\mfab}$   have sup-analytic singularities. 
 \end{cor}

For the case of a Siu function, first note that, while its partial sums  have analytic singularities and converge to $\psi_{\mfab}$, it might be hard to apply Definition~\ref{supa} directly since it involves maximality. Instead, we will use Theorem~\ref{eaeb}. 

\begin{proof} 
 The Green function of $\psi_{\mfab}$ is equal to $G_{\mfab}$ by Corollary~\ref{ggvp} which depends on Theorem~\ref{eaeb}. By definition of sup-analytic singularities, it suffices to show that 
 $\vp = G_{\mfab}$ has sup-analytic singularities. We confirm this by applying Proposition~\ref{sup-an}, (2) to $\vp = G_{\mfab}$: 
 we have the convergence  $$(dd^c h_k)^n (\{0\}) \to (dd^c G_{\mfab})^n (\{0\}) $$ from \cite[Proof of Prop. 5.1]{R13} where $h_k = \frac{1}{k} G_{\mfa_k}$ as was in the definition of $G_{\mfab}$. 
\end{proof}

We remark that $\mfab$ in Proposition~\ref{sup-an}, (3) can be different from the given $\mfab$ as can be seen in the toric case (cf. \cite[(2.9)]{KS19}).

\begin{corollary}\label{supsup}

  Let $\vp$ be a psh weight having sup-analytic singularities. Then there exists a graded system of ideals $\mfab$ such that one (thus every) Siu function $\psi_{\mfab}$ of $\mfab$ is v-equivalent to $\vp$.

\end{corollary}

 \begin{proof}

  Let $G_\vp$ be the Green function of $\vp$. Applying Proposition~\ref{sup-an} (3) for $G_\vp$ in the place of $\vp$, we get $G_{\mfab (\vp)}= G_\vp$ where $\mfab (\vp)$ is defined by \eqref{gdd}. From  Proposition~\ref{vgreen} and Lemma~\ref{GG} (1), we obtain the desired v-equivalence.
 \end{proof}

\begin{corollary}\label{}

 Let $\vp$ be a psh weight. If $\vp$ is tame or, more generally, asymptotically analytic, then it has sup-analytic singularities.

\end{corollary}

 \begin{proof}

 It suffices to assume that $\vp$ is asymptotically analytic. Choosing a bounded hyperconvex domain $D$ for $\vp$, we see that the Green function $G_\vp$ is also asymptotically analytic as follows. Let $(1+\ep)\vp_\ep + O(1) \le \vp \le (1-\ep) \vp_\ep + O(1)$ be the defining condition for $\vp$ where $\vp_\ep$ is with analytic singularities. From Proposition~\ref{basic}, we get

 $$(1+\ep) G_{\vp_\ep} + O(1) \le G_\vp \le (1-\ep) G_{\vp_\ep} + O(1) $$

 \noi where $G_{\vp_\ep} = \vp_\ep + O(1) $ has analytic singularities. Therefore $G_\vp$ is also asymptotically analytic and the condition in Proposition~\ref{sup-an} is satisfied.
 \end{proof}

\medskip

\subsection{Demailly's strong continuity for sup-analytic singularities}

 Now we proceed toward showing that Demailly's strong continuity holds for sup-analytic singularities. We first recall the following results from \cite{R13}. Let $\vp \in W_0$. Let $(\vp_m)_{m \ge 1}$ be the Demailly approximation of $\vp$ (cf. Definition~\ref{5psh}). Recall from \eqref{greeni} that the Green function of $\vp$ on $D$ is denoted by $G_\vp$. In particular, $G_{\vp_m}$ is the Green function of $\vp_m$.

\begin{proposition}\label{max} {\rm \cite[Prop. 4.1, 4.3, Thm. 4.7]{R13}}
Assume $\vp \in MW_0$. Then
\begin{enumerate}
\item
$G_{\vp_{m!}}$ decreases as $m\to \infty$ to the function
$\tilde{G}_\vp:= \inf_m G_{\vp_m}\in \PSH^-(D)$, plurisubharmonic in $D$ and maximal in $D\setminus\{0\}$.
\item $L_n(G_{\vp_{m!}})$ increases as $m\to \infty$ to
$L_n(\tilde{G}_\vp)= \sup_m L_n(G_{\vp_m})$.

\item $\tilde{G}_\vp\ge G_\vp$.

\item $\tilde{G}_\vp= G_\vp$ if and only if $L_n(\tilde{G}_\vp)= L_n(G_\vp)$.

\item $L_n(\tilde{G}_\vp)= L_n(G_\vp)$ if and only if there exists a sequence $\psi_j\in \PSH(D)$ with analytic singularities such that $\psi_j\ge \vp$ and $L_n(\psi_j)\to L_n(\vp)$.

\item If $\tilde{G}_\vp= G_\vp$, then $L_n(G_{\vp_{m}})\to L_n({G}_\vp)$ and $G_{\vp_{m}}\to G_\vp$ in $L^n(D)$.
\end{enumerate}
\end{proposition}

\begin{remark1}

Note that $G_{\tilde{G}_\vp}= \tilde{G}_\vp$. Applying then (5) and (6) to $\tilde{G}_\vp$ instead of $\vp$, we get that, in general,  $L_n(\vp_m) \to L_n(\tilde{G}_\vp)$ and $G_{\vp_{m}}\to \tilde{G}_\vp$ in $L^n(D)$.
\end{remark1}

\medskip
The following can be used to show Demailly's strong continuity by reducing to maximal psh weights. 

\begin{proposition}\label{arb} 
Let $(\vp_m)_{m \ge 1}$ be the Demailly approximation sequence of a psh weight $\vp\in W_0$. Then

\begin{enumerate}
\item$ L_n(\vp_{m_k})\to L_n(\vp)$ as $k\to\infty$ if and only if
$ L_n(G_{\vp_{m_k}})\to L_n(G_\vp)$. 

\item
$L_n(\vp_m) \to L_n(\vp)$ if and only if there exists a sequence $\psi_j\in \PSH(D)$ with analytic singularities such that $\psi_j\ge \vp+O(1)$ and
$L_n(\psi_j)\to L_n(\vp)$.
\end{enumerate}
\end{proposition}

\begin{proof}
(1) Since $\JJ(m\vp) = \JJ(m G_\vp)$ for all $m > 0$ (Proposition~\ref{vgreen}),  we have the following relation, applying Proposition~\ref{basic} and basic properties of the Demailly approximation $\vp \le \vp_m + O(1) = \frac{1}{m} \log \abs{\JJ(m \vp)} +O(1)$:

\begin{align*}
\vp+O(1)\le G_\vp &\le  \vp_m +O(1) = G_{\vp_m} \\
 =  \frac{1}{m} G_{\JJ(m\vp)} &=  \frac{1}{m} G_{\JJ(m G_\vp)}= G_{(G_\vp)_m}= (G_\vp)_m +O(1)
\end{align*}
\noi where the last equality holds since $(G_\vp)_m$ (as the $m$-th Demailly approximant of $G_\vp$) has analytic singularities. Also recall that in our definition and notation, $G_\mfa = G_{\log \abs{\mfa}}$ for an ideal $\mfa$. Hence we have

$$ L_n(\vp) = L_n(G_\vp) \le  L_n(\vp_m)= L_n(G_{\vp_m}) =  L_n((G_\vp)_m),$$
where the first equality is proved in \cite{R06}.

(2) If $\psi_j\ge \vp+O(1)$, then $G_{\psi_j}\ge G_\vp$ and the assertion follows from (1) and Proposition~\ref{max}.
\end{proof}

\medskip

Now we recall a construction from \cite{R13} of the Green function $G_{\mfbb}$ of the asymptotic multiplier ideals $\mfbb$. Denote $H_k=\frac1k G_{\mfb_k}$, then the functions $H_{m!}$ decrease to $G_{\mfbb}\in MW_0$. Since $G_{\mfb_k} \ge G_{\mfa_k}$ for any $k$, we have $G_{\mfbb} \ge G_{\mfab}$.

\begin{proposition}\label{eaeb_gr} For any graded system of  $\mfm$-primary ideals $\mfab$, we have $G_{\mfbb} = G_{\mfab}$.
\end{proposition}

\begin{proof} 

From \cite[Prop. 5.1]{R13}, we have the equivalence : $G_{\mfbb} = G_{\mfab}$ if and only if $e(\mfab) = e(\mfbb)$.  The latter equality  is now given by Theorem~\ref{eaeb}. 
\end{proof}

Now we specify this last result for our choice of the graded system $\mfab$ when we prove the following

\begin{thm}\label{sup-an-conj} Demailly's strong continuity holds true for any $\vp\in W_0$ with sup-analytic singularities.
\end{thm}

\begin{proof}

Given such $\vp$, take  the graded system of ideals $\mfab :=\mfab(G_\vp)$ defined by \eqref{gdd}.   Let $\mfbb$ be the corresponding asymptotic multiplier ideals of $\mfab$.

Then, by Propositions~\ref{sup-an} and \ref{eaeb_gr}, we have $G_{\mfbb} = G_{\mfab(G_\vp)} = G_{G_\vp} = G_\vp$. By the construction, the function $G_{\mfbb}$ is the limit of a decreasing sequence of functions $\psi_m:=H_{m!}$ with analytic singularities such that $L_n(\psi_m)\to L_n(G_{\mfbb})=L_n(\vp)$. Therefore, Propositions~\ref{max} and \ref{arb} imply the convergence  $L_n(\vp_m) \to L_n(\vp)$ where $\vp_m$ is the Demailly approximation of $\vp$.
\end{proof}

  We remark that a maximal psh weight with Demailly’s strong continuity property was called inf-analytic in \cite[Thm. B]{R13}. Hence Theorem~\ref{sup-an-conj} now says that sup-analyticity implies inf-analyticity (cf. \cite[Thm. C]{R13}). 

\medskip

  In the rest of this subsection, we present Proposition~\ref{plusmax}. When combined with Theorem~\ref{sup-an-conj},   this provides an instance of the (unknown) implication : `if Demailly's strong continuity holds true for psh weights $\vp,\psi\in W_0$, then Demailly's strong continuity holds for $\vp+\psi$ and $\max\{\vp,\psi\}$'.

\begin{proposition}\label{plusmax}

If $\vp,\psi\in W_0$ have sup-analytic singularities, then so do $\vp+\psi$ and $\max\{\vp,\psi\}$.
\end{proposition}

\begin{proof}

Let $\vp_j$ and $\psi_j$ be maximal psh weights with analytic singularities, increasing almost everywhere to $G_\vp$ and $G_\psi$, respectively; we can always assume $ \vp_j=G_{\vp_j}\sim c_j\log|f_j|$ and $\psi_j= G_{\psi_j}\sim d_j\log|g_j|$ with rational $c_j$ and $d_j$. Then ${\vp_j}+{\psi_j}$ have analytic singularities as well and increase almost everywhere to $G_\vp+G_\psi\le G_{\vp+\psi}$. The sequence  $u_j:=G_{\vp_j+\psi_j}$ increases almost everywhere to a function $G\in MW_0$ since $\vp_j+\psi_j\le G_{\vp+\psi}$, we have $u_j\le G_{\vp+\psi}$ as well and thus $G\le G_{\vp+\psi}$.

On the other hand, $u_j \ge \vp_j + \psi_j$, and letting $j \to\infty$ we get $G\ge G_{\vp} + G_{\psi}\ge \vp+\psi +O(1)$, so the residual Monge-Amp\`ere mass of $G$ does not exceed that of $\vp +\psi$, and so of $G_{\vp+\psi}$. Applying Lemma~\ref{DP}, we get $G=G_{\vp+\psi}$. By Definition~\ref{supa}, this proves the claim for $\vp+\psi$.
 Similar arguments for $\max\{\vp_j,\psi_j\}$ complete the proof.
\end{proof}

\begin{remark1}

 In general, for two psh weights $u$ and $v$ on a bounded hyperconvex domain, we always have $G_{u+v} \ge G_u + G_v$, but the equality $G_{u+v} = G_u + G_v$ is quite exceptional. It holds, for example, for toric $u$ and $v$ in the unit polydisk, but not necessarily in the ball (unless $G_u=G_v$). A concrete example is given at the end of \cite{R12}.

\end{remark1}

\medskip

\subsection{Analytic approximations from below}

 Let $\vp$ be a psh function. In general, we may call a sequence of psh functions $(\vp_m)_{m \ge 1}$ as a \emph{general analytic approximation} of $\vp$ if each $\vp_m$ has analytic singularities and $\vp_m$ converges to $\vp$ (as $m \to \infty$) in various senses.

 While the Demailly approximation of a psh function (cf. \cite{D11}) can be considered as the `canonical' analytic approximation,  sometimes one may also consider other analytic approximations. In particular, as opposed to the Demailly approximation being an approximation from above, one might as well look for analytic approximations from below, if exist.

 Let $(\vp_m)_{m \ge 1}$ be a general analytic approximation of a psh function $\vp$. In this final subsection, we would  like to address the following question: {\it Given a psh weight $\vp\in W_0$,  how far is the multiplier ideal $\JJ(\vp)$ from the multiplier ideals $\JJ(\vp_m)$ of the approximants?}

Since $\vp$ and the greenification $G_\vp$ are valuatively equivalent, we can always assume here $\vp\in MW_0$. Evidently, for the Demailly approximants $\vp_m$ one always has $\JJ(\vp)\subseteq\JJ(\vp_m)$ for all $m>1$. However it is easy to see that in general one cannot hope that $\JJ(\vp)=\JJ(\vp_m)$ for some $m$ sufficiently big. Indeed, let $\vp=2\log|\mathfrak{m}|$ in $\CC^2$ . By Howald's theorem \cite{Ho}, $\JJ(m\vp)$ is generated by monomials of degree $2m-1$, so $\vp_m=\frac{2m-1}{m} \log|z| +O(1)$. Therefore, $\JJ(\vp_m)=\OO_0$ for all $m>1$ while $\JJ(\vp)=\mathfrak{m}$.

\begin{remark1}

 On the other hand, Demailly \cite[p.64, Remark 3]{D13} showed that the equality $\JJ(\vp) = \JJ((1+\frac{1}{m}) \vp_m)$ holds, using the strong openness theorem of \cite{GZ}. 

\end{remark1}

If, however, $\vp\in W_0$ has sup-analytic singularities, then we have another natural analytic singularities approaching that of $\vp$, or more precisely, that of $G_\vp$.

\begin{proposition} If $\vp\in W_0$ has sup-analytic singularities, then there exists a sequence of psh functions $\psi_k\in \PSH(D)$ with analytic singularities, increasing to $G_\vp$ and having the following property: For every real $m \ge 1$, there exists $k_m \ge 1$ such that
$$ \JJ(m\psi_k)=\JJ(m\vp), \ \forall k\ge k_m .$$
\end{proposition}

\begin{proof}

 For the Green function $G_\vp$, take the ideals $ \mfa_k := \mfa_k(G_\vp)$ by \eqref{gdd}. As was in the definition of $G_{\mfab}$, let $h_k=\frac1{k} G_{\mfa_{k}}$, where $G_{\mfa_k}$ are the Green functions of the ideals $ \mfa_k$. Then $h_{k!}$ increase almost everywhere to $G_{\mfab}$ which is equal to $G_\vp$ by Proposition~\ref{sup-an} (3).
 
 Now we use the same argument as in Lemma~\ref{GG} (1): by  the version of the strong openness theorem for increasing sequences (see, e.g., \cite{Hi17}, \cite{Le17}), we have $\JJ(h_{k!})=\JJ(G_\vp)$ for all $k$ sufficiently big. Moreover, applying this to the functions $m h_{k!}$ increasing to $mG_\vp$, we get $\JJ(mh_{k!})=\JJ(mG_\vp)$ for all $k\ge k_m$. Since $\JJ(m\vp)= \JJ(mG_\vp)$ for all $m \ge 1$ by Proposition~\ref{vgreen}, the proof is complete when we take $\psi_k = h_{k!}$. 
\end{proof}

\section{Appendix  by S\'ebastien Boucksom: Asymptotic multiplicities of \\ graded systems of ideals }

 \footnote{\;\; S\'ebastien Boucksom :  CNRS–CMLS, 
\'Ecole Polytechnique, F-91128 Palaiseau Cedex, France \\
 \; e-mail: sebastien.boucksom@polytechnique.edu}

 In this  algebraic appendix, the proof of Theorem~\ref{eaeb} is given.  It is based on the intersection theory of nef Weil b-divisors developed in \cite{BFJ}, \cite{BFF}.

 Let $X$ be an irreducible smooth complex variety.   Recall from \cite[\S 1.2]{BFJ} that,  given an ideal sheaf $\mfa$ on $X$, a b-divisor $Z(\mfa)$ on $X$ is defined by letting $\ord_E Z(\mfa) = - \ord_E (\mfa)$ for every prime divisor $E$ lying over $X$ (so that $Z(\mfa) \le 0$). It is a Cartier b-divisor determined on the normalized blowup of $X$ along $\mfa$ (cf. \cite[Example~1.4]{BFF}). We remark that the notion of a b-divisor was first introduced by Shokurov (cf. \cite{Sh03}) in the context of birational geometry and the minimal model program. See \cite[\S 1]{BFF} for an exposition of basic properties of b-divisors.

 Now recall that a graded system of ideals in $\OO_X$  is a sequence of ideal sheaves $(\mfa_m)_{m \ge 1}$ satisfying $\mfa_m \cdot \mfa_k \subset \mfa_{m+k}$ for all $m, k \ge 1$. Given $\mfab$,  an $\RR$-Weil b-divisor $\displaystyle Z(\mfab) := \lim_{m \to \infty} \frac{1}{m} Z(\mfa_m)  $ is defined  (see \cite[Prop. 2.1]{BFF}) where the coefficient-wise limit exists due to the fact that $ \displaystyle \frac{1}{l} Z(\mfa_l) \le \frac{1}{m} Z(\mfa_m)$ for every $m$ divisible by $l$. Note that the limit b-divisor $Z(\mfab)$ is a nef b-divisor, but not necessarily b-Cartier, unlike each $Z(\mfa_m)$.

 Another ingredient we need is the following intersection number of nef $\RR$-Weil b-divisors $Z_1, \ldots, Z_n$ due to \cite[Def. 4.3]{BFJ}  where $n = \dim X$:

 \begin{equation*}
  \langle Z_1, \ldots, Z_n \rangle := \inf \{ \langle W_1, \ldots, W_n \rangle : W_j \ge Z_j  \} \in [-\infty, 0]
 \end{equation*}

\noi where $W_j \; (1 \le j \le n)$ is a nef $\RR$-Cartier b-divisor.

\begin{proof}[Proof of Theorem~\ref{eaeb}]

 First we recall the following relation between the intersection numbers and mixed multiplicities.
 For $\mfm$-primary ideals $\mfa_1, \ldots, \mfa_n$ on $X$, we have  (cf. \cite[(4.3)]{BFJ})
 $$- \langle  Z(\mfa_1), \cdots, Z(\mfa_n)  \rangle = e( \mfa_1, \cdots, \mfa_n)$$ where the RHS is the mixed multiplicity of the ideals.  The equality is checked on a birational model $X' \to X$ which dominates the blowup of each $\mfa_i$   as in \cite[p.92]{L}. Thus we have $e(\mfa) =   - \langle Z(\mfa ) \rangle^n$ for a $\mfm$-primary ideal $\mfa$ where $e(\mfa)$ is the Samuel multiplicity.

On the other hand, the limit b-divisor $Z(\mfab)$ is also equal to the increasing (i.e. nondecreasing) limit
$ \displaystyle Z(\mfab) = \lim_{m \to \infty}   \frac{1}{m!}  Z (\mfa_{m!})$ from \cite[(4.17)]{BFF}. Therefore we have the convergence
$$\frac{1}{k^n} e(\mfa_k) = -\langle \frac{1}{k} Z(\mfa_k), \cdots, \frac{1}{k} Z(\mfa_k) \rangle \to  -\langle Z(\mfab), \cdots, Z(\mfab) \rangle $$ first for $k=m!$ by \cite[Thm. A.1]{BFF} and then for general $k$ due to the fact that the limit $\displaystyle e(\mfab) = \lim_{k \to \infty} \frac{1}{k^n} e(\mfa_k) $ exists by \cite[Cor. 1.5]{M02}. In other words,  we have $$e(\mfab) = -\langle Z(\mfab), \cdots, Z(\mfab) \rangle .$$

 With similar arguments for the asymptotic multiplier ideals $\mfbb$, we see that $e(\mfbb) =  \lim_{k \to \infty} \frac{1}{k^n} e(\mfb_k)$ (the limit exists due to \cite[Cor. 2.3]{M02}) satisfies

 $$ e(\mfbb) = -\langle Z(\mfbb), \cdots, Z(\mfbb) \rangle $$
 \noi where  $\displaystyle Z(\mfbb) := \lim_{k \to \infty}   \frac{1}{k}  Z (\mfb_k) $, which is  a decreasing limit coefficientwise. Here \cite[Prop. 4.4]{BFJ} was used to ensure that the intersection number is continuous along the decreasing sequence $\displaystyle \frac{1}{k}  Z (\mfb_k) $.

 Finally we have the relation  $Z(\mfab) = Z(\mfbb)$ from  \cite[Prop. 2.13 (ii)]{JM}.  Combining with the above relations, it follows that $$e(\mfa_\bullet) = -\langle Z(\mfab), \cdots, Z(\mfab) \rangle = -\langle Z(\mfbb), \cdots, Z(\mfbb) \rangle = e(\mfb_\bullet). $$
\end{proof}

\medskip

\footnotesize

\bibliographystyle{amsplain}

\qa
\\

\medskip

\normalsize

\noi \textsc{Dano Kim}

\noi Department of Mathematical Sciences and Research Institute of Mathematics

\noi Seoul National University, 08826  Seoul, Korea

\noi Email address: kimdano@snu.ac.kr

\qa
\qa

\noi

\noi \textsc{Alexander Rashkovskii}

\noi Tek/Nat, University of Stavanger, 4036 Stavanger, Norway

\noi Email address: alexander.rashkovskii@uis.no
\noi

\end{document}